\documentclass[sn-mathphys-num]{sn-jnl}

\usepackage{graphicx}%
\usepackage{multirow}%
\usepackage{amsmath,amssymb,amsfonts}%
\usepackage{amsthm}%
\usepackage{mathrsfs}%
\usepackage[title]{appendix}%
\usepackage{xcolor}%
\usepackage{textcomp}%
\usepackage{manyfoot}%
\usepackage{booktabs}%
\usepackage{algorithm}%
\usepackage{algorithmicx}%
\usepackage{algpseudocode}%
\usepackage{listings}%
\usepackage{mathtools}

\theoremstyle{definition}
\newtheorem{thm}{Theorem}[section]
\newtheorem{rem}[thm]{Remark}
\newtheorem{defi}[thm]{Definition}
\newtheorem{lem}[thm]{Lemma}
\newtheorem{cor}[thm]{Corollary}
\newtheorem{prop}[thm]{Proposition}
\newtheorem{ex}[thm]{Example}

\theoremstyle{thmstyleone}%
\theoremstyle{thmstyletwo}%

\theoremstyle{thmstylethree}%

\newcommand{\boldf}{\mathbf{f}}
\newcommand{\bh}{\mathbf{h}}
\newcommand{\bk}{\mathbf{k}}
\newcommand{\bm}{\mathbf{m}}
\newcommand{\bn}{\mathbf{n}}
\newcommand{\bx}{\mathbf{x}}
\newcommand{\by}{\mathbf{y}}
\newcommand{\bz}{\mathbf{z}}

\newcommand{\bbE}{\mathbb{E}}

\newcommand{\bbN}{\mathbb{N}}

\newcommand{\bbP}{\mathbb{P}}

\newcommand{\bbR}{\mathbb{R}}

\newcommand{\bbZ}{\mathbb{Z}}

\newcommand{\cH}{\mathcal{H}}

\newcommand{\cT}{\mathcal{T}}

\newcommand{\es}{\emptyset}
\newcommand{\Gbar}{\overline{G}}

\newcommand{\ra}{\rightarrow}

\newcommand{\sm}{\setminus}
\newcommand{\sbse}{\subseteq}

\newcommand{\iu}{\mathrm{i}}

\newcommand{\Le}{L_2^{\text{extr}}}
\newcommand{\Lp}{L_2^{\text{per}}}
\newcommand{\Ls}{L_2^{\text{star}}}

\newcommand{\prob}{\operatorname{prob}}

\newcommand{\Var}{\operatorname{Var}}

\newcommand{\epsi}{\epsilon}
\newcommand{\exT}{\epsi^{}_\cT}
\newcommand{\etp}{\eta^\text{p}_2}
\newcommand{\eoe}{\eta^\text{e}_1}
\newcommand{\ete}{\eta^\text{e}_2}
\newcommand{\dx}{\text{d}}
\newcommand{\ETe}{E^\text{extr}_\cT}
\newcommand{\ETp}{E^\text{per}_\cT}

\begin{document}

\title[Article Title]{On the $L_2$-discrepancy of Latin hypercubes}


\author*{\fnm{Nicolas} \sur{Nagel}\footnote{\textsc{Department of Mathematics, University of Technology Chemnitz, Germany} \\
		\texttt{nicolas.nagel@math.tu-chemnitz.de} \\
		ORCID iD: 0009-0004-3362-3543}}


%

\abstract{We investigate $L_2$-discrepancies of what we call weak Latin hypercubes. In this case it turns out that there is a precise equivalence between the extreme and periodic $L_2$-discrepancy which follows from a much broader result about generalized energies for weighted point sets.
	
Motivated by this we study the asymptotics of the optimal $L_2$-discrepancy of weak Latin hypercubes. We determine asymptotically tight bounds for $d \geq 3$ and even the precise (dimension dependent) constant in front of the dominating term for $d \geq 4$.}


\maketitle

\section{Introduction}

Given a finite set $X = \{\bx^1, ..., \bx^N\} \sbse [0, 1)^d$, the notion of discrepancy measures how uniformly distributed $X$ is over the domain $[0, 1)^d$. The formal definition goes as follows. For $x, y \in [0, 1)$ consider the \textbf{(periodic) interval}
$$
[x, y) \coloneqq \begin{cases}
	\{z \in [0, 1): x \leq z < y\} & , x < y \\
	\{z \in [0, 1): x \leq z \text{ or } z < y\} & , x \geq y
\end{cases}
$$
(that is, it is the usual interval if $x < y$, otherwise we interprete $[0, 1)$ as a torus and the periodic interval wraps around $1$) and for $\bx = (x_1, \dots, x_d), \by = (y_1, \dots, y_d) \in [0, 1)^d$ the corresponding $d$-dimensional version $[\bx, \by) \coloneqq \prod_{k=1}^d [x_k, y_k)$. The Lebesgue measure of such a set is given by
$$
|[\bx, \by)| = \prod_{\substack{k=1 \\ x_k < y_k}}^d (y_k-x_k) \prod_{\substack{k=1 \\ x_k \geq y_k}}^d (1-x_k+y_k).
$$
This enables us to define the \textbf{discrepancy function}
$$
D_X(\bx, \by) \coloneqq \#(X \cap [\bx, \by)) - N |[\bx, \by)|
$$
of $X \sbse [0, 1)^d, \#X = N$ (note that some authors use the normalized $N^{-1} D_X(\bx, \by)$ as the discrepancy function). The discrepancy of $X$ is now given by a norm of $D_X$. We will be interested in the $L_2$-norm, of which we can define three types, depending on the domain of allowed boxes $[\bx, \by)$. The \textbf{star/extreme/periodic $L_2$-discrepancy} are given respectively by
\begin{align*}
	\Ls(X)^2 & \coloneqq \int_{[0, 1)^d} D_X(\mathbf{0}, \by)^2 \, \dx\by, \\
	\Le(X)^2 & \coloneqq \int_{[0, 1)^d} \int_{\by > \bx} D_X(\bx, \by)^2 \, \dx\by \, \dx\bx, \\
	\Lp(X)^2 & \coloneqq \int_{[0, 1)^d} \int_{[0, 1)^d} D_X(\bx, \by)^2 \, \dx\by \, \dx\bx,
\end{align*}
where $\by > \bx$ is supposed to be understood componentwise. Thus, star discrepancy only considers boxes anchored at $\mathbf{0}$, extreme discrepancy takes all non-periodic intervals into account, while periodic discrepancy also considers periodic intervals on the torus. Clearly $\Le(X) \leq \Lp(X)$ as the integral for the periodic $L_2$-discrepancy goes over a larger domain.

It is easy to calculate the integrals explicitely in dependence of $X$ and we get the formulas (often called Warnock-type formulas, see \cite{HKP20, MC94, Nie73, NW08, War72})
\begin{align} \label{eq:Warnock}
	\begin{split}
		\Ls(X)^2 & = \frac{N^2}{3^d} - \frac{N}{2^{d-1}} \sum_{\bx \in X} \prod_{k=1}^d (1-x_k^2) + \sum_{\bx, \by \in X} \prod_{k=1}^d (1-\max\{x_k, y_k\}), \\
		\Le(X)^2 & = \frac{N^2}{12^d} - \frac{N}{2^{d-1}} \sum_{\bx \in X} \prod_{k=1}^d x_k(1-x_k) + \sum_{\bx, \by \in X} \prod_{k=1}^d (\min\{x_k, y_k\} - x_k y_k), \\
		\Lp(X)^2 & = -\frac{N^2}{3^d} + \sum_{\bx, \by \in X} \prod_{k=1}^d \left(\frac{1}{2} - |x_k-y_k| + |x_k-y_k|^2\right).
	\end{split}
\end{align}
The importance of discrepancy comes from its application to numerical integration, in particular quasi Monte Carlo methods. Approximating an integral via
$$
\int_{[0, 1)^d} f(\bx) \, \dx\bx \approx \frac{1}{N} \sum_{\bx \in X} f(\bx)
$$
with $\#X = N$, the relative worst case error can usually be expressed by some notion of discrepancy of the point set $X$. We recall the precise setup from \cite{HO16} (adapted to our setting). Consider the Sobolev space $H^1[0, 1]$ consisting of absolutely continuous functions $f: [0, 1] \ra \bbR$ with $f \in L_2[0, 1]$, $f(0) = f(1)$ (periodicity) and
$$
\|f\|_{H^1[0, 1]}^2 \coloneqq 3 \left(\int_0^1 f(x) \, \dx x\right)^2 + \frac12 \int_0^1 f'(x)^2 \, \dx x < \infty,
$$
where $f'$ denotes the weak derivative. Furthermore, let $H^1[0, 1]^d$ be given by the $d$-fold tensor product of $H^1[0, 1]$ with itself. The norm of $H^1[0, 1]$ induces a norm on $H^1[0, 1]^d$ and it turns out that
\begin{align*}
	\Lp(X) = N \cdot \sup\limits_{\|f\|_{H^1[0, 1]^d} \leq 1} \left|\int_{[0, 1)^d} f(\bx) \, \dx\bx - \frac{1}{N} \sum_{\bx \in X} f(\bx)\right|.
\end{align*}
Indeed, this is a consequence of $H^1[0, 1]^d$ being a \textit{reproducing kernel Hilbert space} (see \cite{BT11} for the general theory) with kernel given by
$$
K(\bx, \by) \coloneqq \prod_{k=1}^d \left(\frac{1}{2} - |x_k-y_k| + |x_k-y_k|^2\right)
$$
where $\bx, \by \in [0, 1]^d$, for which the integration error can be expressed easily (see \cite{NW08} formula (9.31)). Similar statements can be made for the star and extreme $L_2$-discrepancy, see the proof of Theorem 5 in \cite{HKP20}. Minimizing $\Lp(X)$ is thus equivalent to minimizing the error of the quasi Monte Carlo rule for integration. For further material on quasi Monte Carlo methods and related topics see \cite{DP19, Ham13, KN12, Nie92}.

Another related notion to periodic $L_2$-discrepancy is the so called \emph{diaphony} as introduced in \cite{Zin76}, also see \cite{Pil23}, given by 
\begin{align} \label{eq:diaphony}
	F_N(X)^2 \coloneqq \sum_{\bh \in \bbZ^d \sm \{\mathbf{0}\}} \left(\prod_{k=1}^d \max\{1, |h_k|\}\right)^{-2} \frac1N \left|\sum_{\bx \in X} \exp(2\pi\iu \bh \cdot \bx)\right|^2,
\end{align}
where $\iu^2=-1$ is the imaginary unit and $\bh \cdot \bx$ denotes the usual inner product. In fact, the periodic $L_2$-discrepancy has a very similar representation given by \cite{HKP20, Pil23}
\begin{align} \label{eq:periodisc_disc_Fourier}
	\Lp(X)^2 = \frac1{3^d} \sum_{\bh \in \bbZ^d \sm \{\mathbf{0}\}} \left(\prod_{k=1}^d \max\left\{1, \frac{2\pi}{\sqrt{6}} |h_k|\right\}\right)^{-2} \left|\sum_{\bx \in X} \exp(2\pi\iu \bh \cdot \bx)\right|^2.
\end{align}
For a connection between diaphony and worst case errors of quasi Monte Carlo integration and further references see \cite{Pil13}.

The optimal (asymptotic) behaviour of the three notions of discrepancy as introduced above is of particular interest and has been studied for a long time. It is a classical result of Roth \cite{Rot54} that
$$
\Ls(X) \gtrsim (\log N)^{(d-1)/2},
$$
which was generalized in \cite{HKP20} for the other notions of discrepancy to
\begin{align} \label{ineq:roth}
	\Lp(X) \geq \Le(X) \gtrsim (\log N)^{(d-1)/2}
\end{align}
(also see \cite{HMOU16} and Section 8.7 in \cite{DTU18} for the general setting of quasi Monte Carlo integration over Besov spaces). All bounds are known to be tight up to the multiplicative constant \cite{Che80, CS02, Dav56, GSY15, Rot80} (also see \cite{Skr06} and further references therein). The same bound (up to the implicite constant) holds for the \emph{$L_p$-discrepancy}, that is the $p$-norm of the discrepancy function, for any $1 < p < \infty$ (see \cite{Sch77} and \cite{Bil14} for a detailed survey on the case of star $L_p$-discrepancy as well as \cite{KP22} for extreme and periodic $L_p$-discrepancy). The asymptotics in the case of $p=\infty$ is a major open problem in discrepancy theory and there is still a sizable gap between the best know upper and lower bounds, see \cite{BLV08} for the best known results in this directions.

Two common examples of two-dimensional point sets having periodic $L_2$-discrepancy of asymptotically optimal order are the Hammersley/van der Corput point sets \cite{Cor35, Ham13} and the Fibonacci lattice \cite{BTY12} (which are particular examples of rational/integration lattices \cite{HKP20}). Note that Fibonacci lattices are also of optimal order for star and extreme $L_2$-discrepancy, while the Hammersley point set is only optimal for extreme $L_2$-discrepancy. For star $L_2$-discrepancy it is of worse order than Roth's bound, see Theorem 8 in \cite{HKP20} and Theorem 9 in the same paper or \cite{FP11, HZ69} for an improved version of the Hammersley set. In \cite{HKP20} it was observed that Hammersley sets and rational lattices $X \sbse [0, 1)^2$ consisting of $N$ points fulfil the precise relation

\begin{align} \label{eq:Lp_Le_two_dim}
	\Lp(X)^2 - 4 \Le(X)^2 = \frac{N^2+1}{18N^2}.
\end{align}
Both of these classes of point sets are particular examples of what we will call \textit{permutation sets}.

\begin{defi} \label{def:permutation_set}
	Let $\sigma: \{0, 1, \dots, N-1\} \ra \{0, 1, \dots, N-1\}$ be a permutation, define the \textbf{permutation set}
	$$
	X(\sigma) \coloneqq \left\{\left(\frac{m}{N}, \frac{\sigma(m)}{N}\right): m=0, 1, \dots, N-1\right\}.
	$$
\end{defi}

It turns out that \eqref{eq:Lp_Le_two_dim} holds even more general for permutation sets.

\begin{thm} \label{thm:Lp_Le_permutation}
	For any permutation $\sigma: \{0, 1, \dots, N-1\} \ra \{0, 1, \dots, N-1\}$ it holds
	\begin{align*} 
		\Lp(\sigma)^2 - 4 \Le(\sigma)^2 = \frac{N^2+1}{18N^2}.
	\end{align*}
\end{thm}

We can generalize this to arbitrary dimensions as follows. Note that among the \textit{discretized torus} $G \coloneqq \frac{1}{N} \{0, 1, \dots, N-1\}^2$ permutation sets are precisely those sets $X \sbse G$ with $\#X = N$ and containing exactly one point from every row and column. Generalizing this to the $d$-dimensional discretized torus we obtain the notion of a \textit{(weak) Latin hypercube}, which will be made precise in Section \ref{sec:weighted_discrepancy} and an analogous result to Theorem \ref{thm:Lp_Le_permutation} will be given in Theorem \ref{cor:Lp_Le_latin_hypercubes}.

The remaining paper will be structured as follows. Motivated by trying to find a higher dimensional analogue of Theorem \ref{thm:Lp_Le_permutation} we first introduce a weighted notion of $L_2$-discrepancy and, even more general, energies closely related to extreme and periodic $L_2$-discrepancy. There we will prove a central and general result in Proposition \ref{thm:expression_for_excess}. After that, we apply this to the case of weak Latin hypercubes (a notion that we will introduce in Section \ref{sec:weighted_discrepancy}) giving us Theorem \ref{cor:Lp_Le_latin_hypercubes} for (unweighted) $L_2$-discrepancy, the case $d=2$ being Theorem \ref{thm:Lp_Le_permutation}. A quick look at this result also gives us a lower bound on the periodic $L_2$-discrepancy for weak Latin hypercubes. Motivated by this, we introduce another perspective on our general notion of energy (Theorem \ref{thm:svd_energy}) giving us lower bounds (Theorem \ref{cor:lower_bound_latin_hypercube_discrepancy}) for both periodic and extreme $L_2$-discrepancy of weak Latin hypercubes. We finish by considering random weak Latin hypercubes. While their periodic $L_2$-discrepancy itself seems hard to handle over all, we can still make some statistical statements, determining their expectation in Theorem \ref{cor:expectation_discrepancy}. This gives us upper bounds for the optimal periodic $L_2$-discrepancy of weak Latin hypercubes that match the lower bound asymptotically for $d \geq 3$ and even in the precise dominant behaviour (asymptotically and in the constant factor) for $d \geq 4$.

\section{Discrepancy of weighted point sets and energies} \label{sec:weighted_discrepancy}

A \textbf{weighted point set} $(X, w)$ (called \emph{net} in \cite{Lev95, Lev99}) is given by a finite set $X \sbse [0, 1)^d$ and a function $w: X \ra \bbR$. Note that we allow the weights to be negative. The \textbf{discrepancy function} of $(X, w)$ is given by
$$
D_{(X, w)}(\bx, \by) \coloneqq \sum_{\substack{\bz \in X \\ \bx \leq \bz < \by}} w(\bz) - \sum_{\bz \in X} w(\bz) |[\bx, \by)|
$$
and the corresponding notions of \textbf{star/extreme/periodic $L_2$-discrepancy of weighted point sets} by
\begin{align*}
	\Ls(X, w)^2 & \coloneqq \int_{[0, 1)^d} D_{(X, w)}(0, \by)^2 \, \dx\by, \\
	\Le(X, w)^2 & \coloneqq \int_{[0, 1)^d} \int_{\by > \bx} D_{(X, w)}(\bx, \by)^2 \, \dx\by \, \dx\bx, \\
	\Lp(X, w)^2 & \coloneqq \int_{[0, 1)^d} \int_{[0, 1)^d} D_{(X, w)}(\bx, \by)^2 \, \dx\by \, \dx\bx.
\end{align*}
It is straightforward to determine Warnock-type formulas for these discrepancies. The proofs are analogous to the unweighted case, so we will omit them here.

\begin{prop} \label{prop:weighted_Warnock}
	Let $(X, w)$ be a weighted point set in $[0, 1)^d$, then it holds
	\begin{align*}
		\Ls(X, w)^2 & = \frac{\left(\sum_{\bx \in X} w(\bx)\right)^2}{3^d} - \frac{\sum_{\bx \in X} w(\bx)}{2^{d-1}} \sum_{\bx \in X} w(\bx) \prod_{k=1}^d (1-x_k^2) \\
		& + \sum_{\bx, \by \in X} w(\bx) w(\by) \prod_{k=1}^d (1-\max\{x_k, y_k\}),
	\end{align*}
	\begin{align*}
		\Le(X, w)^2 & = \frac{\left(\sum_{\bx \in X} w(\bx)\right)^2}{12^d} - \frac{\sum_{\bx \in X} w(\bx)}{2^{d-1}} \sum_{\bx \in X} w(\bx) \prod_{k=1}^d x_k(1-x_k) \\
		& + \sum_{\bx, \by \in X} w(\bx) w(\by) \prod_{k=1}^d (\min\{x_k, y_k\} - x_k y_k),
	\end{align*}
	\begin{align*}
		& \Lp(X, w)^2 \\
		= & -\frac{\left(\sum_{\bx \in X} w(\bx)\right)^2}{3^d} + \sum_{\bx, \by \in X} w(\bx) w(\by) \prod_{k=1}^d \left(\frac{1}{2} - |x_k-y_k| + |x_k-y_k|^2\right).
	\end{align*}
\end{prop}


Note the similarities to \eqref{eq:Warnock}. The formulas in the proposition can be interpreted as a sort of energy for particles in $[0, 1)^d$. Here $X$ describes the positions of the particles and the weight $w$ some property of each particle (think charge). Using this interpretation, we can actually generalize it even further by taking these Warnock-type formulas as a starting point and using different potentials. Since we will only be interested in extreme and periodic discrepancy, we will discard the star discrepancy from now on. We start with the following definition.

\begin{defi} \label{def:energy_triple}
	An \textbf{energy triple} $\cT$ is given by three functions $(\etp, \eoe, \ete)$ having the following properties:
	\begin{itemize}
		\item [(i)] $\etp: [0, 1)^2 \ra \bbR$ is symmetric (i.e. $\etp(s, t) = \etp(t, s)$ for all $s, t \in [0,1 )$) and for all $M \in \bbN$ and all $m \in \{0, 1, \dots, N-1\}$ the sum
		$$
		\sum_{n=0}^{M-1} \etp\left(\frac{m}{M}, \frac{n}{M}\right)
		$$
		is a constant $T = T(M)$ independent on $m$ (but can depend on $M$). We will set $\eta(s) \coloneqq \etp(s, 0)$.
		
		\item [(ii)] $\eoe: [0, 1) \ra \bbR, \eoe(s) \coloneqq \eta(0) - \eta(s)$.
		
		\item [(iii)] $\ete: [0, 1)^2 \ra \bbR, \ete(s, t) \coloneqq \eta(0)-\eta(s)-\eta(t)+\etp(s, t)$.
	\end{itemize}
\end{defi}

An energy triple should be seen as a generalization of the terms involved in the expression for extreme and periodic $L_2$-discrepancy. Here $\etp$ is the potential between two points analogous to the periodic $L_2$-discrepancy and $\eoe$ and $\ete$ are generalizations of the potentials involved in the extreme $L_2$-discrepancy ($\eoe$ gives the interaction of a particle with the boundary of $[0, 1)^d$ and $\ete$ the interacion among particle pairs). The precise connection will be given in Example \ref{ex:discrepancy_as_energy} via Definition \ref{def:F_energy}.

Algebraically, these notions are defined in a way as to get a statement as general as possible in Proposition \ref{thm:expression_for_excess} (which will show that under these relations, the quantities defined in Definition \ref{def:F_energy}, generalizing $L_2$-discrepancy, are related in a way similar to the one given in Theorem \ref{thm:Lp_Le_permutation}). Clearly, an energy triple is determined uniquely by the function $\etp$. There are actually a lot of examples of energy triples we can construct thanks to the following.

\begin{prop} \label{prop:energy_triple_canonical_construction}
	Let $\eta: [0, 1) \ra \bbR$ be a function with $\eta(s) = \eta(1-s)$ for all $s \in (0, 1)$. Then $\etp(s, t) = \eta(|s-t|)$ defines an energy triple.
\end{prop}

The function $\eta$ is then also given as in Definition \ref{def:energy_triple} (i), but note that not all energy triples are of this form (for example $\etp(s, t) = \sin(2\pi s) \sin(2\pi t)$ would give an energy triple according to the above definition, with $T \equiv 0$, but not of this form). If an energy triple is of this form we will say it comes from the \textbf{canonical construction via $\eta$}.

\begin{proof} [Proof of Proposition \ref{prop:energy_triple_canonical_construction}]
	By definition it is clear that $\etp$ is symmetric, so it remains to verify the sum property. Indeed, let $m \in \{0, 1, \dots, M-1\}$ and note that
	\begin{align*}
		\sum_{n=0}^{M-1} \etp\left(\frac{m}{M}, \frac{n}{M}\right) & = \sum_{n=0}^{M-1} \eta\left(\frac{|m-n|}{M}\right) = \sum_{n=0}^{m-1} \eta\left(\frac{m-n}{M}\right) + \sum_{n=m}^{M-1} \eta\left(\frac{n-m}{M}\right) \\
		& = \sum_{n=0}^{m-1} \eta\left(1 - \frac{m-n}{M}\right) + \sum_{n'=0}^{M-n-1} \eta\left(\frac{n'}{M}\right) \\
		& = \sum_{n=0}^{m-1} \eta\left(\frac{n + (M-m)}{M}\right) + \sum_{n'=0}^{M-m-1} \eta\left(\frac{n'}{M}\right) \\
		& = \sum_{n'=M-m}^{M-1} \eta\left(\frac{n'}{M}\right) + \sum_{n'=0}^{M-m-1} \eta\left(\frac{n'}{M}\right) = \sum_{n'=0}^{M-1} \eta\left(\frac{n'}{M}\right)
	\end{align*}
	is independent of $m$ ($n'$ denotes where we relabeled the dummy variable $n$).
\end{proof}

As a shorthand for $\bx, \by \in [0, 1)^d$ we set
$$
\etp(\bx, \by) \coloneqq \prod_{k=1}^d \etp(x_k, y_k), \, \eoe(\bx) \coloneqq \prod_{k=1}^d \eoe(x_k), \, \ete(\bx, \by) \coloneqq \prod_{k=1}^d \ete(x_k, y_k).
$$
From now on fix $M, d \in \bbN, d \geq 2$. In what follows we will consider the \textbf{$d$-dimensional discretized torus} $G \coloneqq \frac{1}{M} \{0, 1, \dots, M-1\}^d \sbse [0, 1)^d$. The following is an abstraction of the formulas for extreme and periodic $L_2$-discrepancies of weighted point sets as given in Proposition \ref{prop:weighted_Warnock}.

\begin{defi} \label{def:F_energy}
	Given an energy triple $\cT = (\etp, \eoe, \ete)$ we define the \textbf{extreme/periodic $\cT$-energy} of $w: G \ra \bbR$ by
	\begin{align*}
		\ETe(w) \coloneqq - 2 \left(\sum_{\bx \in G} w(\bx)\right) \sum_{\bx \in G} w(\bx) \eoe(\bx) + \sum_{\bx, \by \in G} w(\bx) w(\by) \ete(\bx, \by)
	\end{align*}
	and
	\begin{align*}
		\ETp(w) \coloneqq \sum_{\bx, \by \in G} w(\bx) w(\by) \etp(\bx, \by)
	\end{align*}
	respectively.
\end{defi}

\begin{ex} \label{ex:discrepancy_as_energy}
	As an example consider the canonically constructed energy triple $\cT = (\etp, \eoe, \ete)$ via the function $\eta(s) = \frac{1}{2} - s + s^2$ (clearly fulfilling $\eta(s) = \eta(1-s)$). We thus have
	$$
	\etp(s, t) = \frac{1}{2} - |s-t| + |s-t|^2, \eoe(s) = s(1-s), \ete(s, t) = 2\left(\min\{s, t\} - st\right).
	$$
	Comparing Proposition \ref{prop:weighted_Warnock} and Definition \ref{def:F_energy} we see that for all $w: G \ra \bbR$ it holds
	$$
	\Lp(G, w)^2 = -\frac{\left(\sum_{\bx \in G} w(\bx)\right)^2}{3^d} + \ETp(w)
	$$
	and
	$$
	\Le(G, w)^2 = \frac{\left(\sum_{\bx \in G} w(\bx)\right)^2}{12^d} + 2^{-d} \ETe(w).
	$$
	Thus $\cT$-energy can be seen as a generalization of discrepancy of weights on $G$.
\end{ex}

We now make some preparations for our central proposition in this section. We start with another definition.

\begin{defi} \label{def:excess}
	Let $\cT$ be an energy triple and $w: G \ra \bbR$. Define its \textbf{excess} by $\exT(w) \coloneqq \ETp(w) - \ETe(w)$.
\end{defi}

Note that for the energy triple given by periodic and extreme $L_2$-discrepancy as in Example \ref{ex:discrepancy_as_energy} and for $d=2$, the excess is (up to an additive constant) the term on the left-hand side of the expression in Theorem \ref{thm:Lp_Le_permutation}. For $r \in \bbR$ consider the weight $w_r: G \ra \bbR$ given by $w_r(\bx) = \frac{r}{M}$.

\begin{lem} \label{lem:excess_of_constant_weight}
	$\exT(w_r) = M^{d-2}\left(T^d + (M\eta(0)-T)^d\right) r^2$ (recall Definition \ref{def:energy_triple} (i) for the definitions of $T$ and $\eta$).
\end{lem}

\begin{proof}
	We need to compute $\ETp(w_r)$ and $\ETe(w_r)$. Write
	$$
	\ETp(w_r) = \frac{r^2}{M^2} \sum_{\bx, \by \in G} \etp(\bx, \by) = \frac{r^2}{M^2} \sum_{\bx, \by \in G} \prod_{k=1}^d \etp(x_k, y_k)
	$$
	and note that the sum factorizes as
	$$
	\ETp(w_r) = \frac{r^2}{M^2} \prod_{k=1}^d \left(\sum_{m_k, n_k = 0}^{M-1} \etp\left(\frac{m_k}{M}, \frac{n_k}{M}\right)\right) = \frac{r^2}{M^2} \left(\sum_{m, n = 0}^{M-1} \etp\left(\frac{m}{M}, \frac{n}{M}\right)\right)^d.
	$$
	By Definition \ref{def:energy_triple} (i) we have
	$$
	\sum_{m=0}^{M-1} \sum_{n=0}^{M-1} \etp\left(\frac{m}{M}, \frac{n}{M}\right) = \sum_{m=0}^{M-1} T = MT
	$$
	and we get
	$$
	\ETp(w_r) = M^{d-2} T^d r^2.
	$$
	As for $\ETe(w_r)$ we have
	\begin{align*}
		\ETe(w_r) & = - 2 rM^{d-1} \cdot \frac{r}{M}\sum_{\bx \in G} \eoe(\bx) + \frac{r^2}{M^2} \sum_{\bx, \by \in G} \ete(\bx, \by) \\
		& = - 2 r^2M^{d-2} \sum_{\bx \in G} \prod_{k=1}^d \eoe(x) + \frac{r^2}{M^2} \sum_{\bx, \by \in G} \prod_{k=1}^d \ete(x_k, y_k)
	\end{align*}
	and factorization again yields
	$$
	\sum_{\bx \in G} \prod_{k=1}^d \eoe(x_k) = \prod_{k=1}^d \left(\sum_{m_k=0}^{M-1} \eta(0) - \eta\left(\frac{m_k}{M}\right)\right) = (M\eta(0) - T)^d
	$$
	and
	\begin{align*}
		\sum_{\bx, \by \in G} \prod_{k=1}^d \ete(x_k, y_k) & = \prod_{k=1}^d \left(\sum_{m_k, n_k = 0}^{M-1} \left[\eta(0) - \eta\left(\frac{m_k}{M}\right) - \eta\left(\frac{n_k}{M}\right) + \etp\left(\frac{m_k}{M}, \frac{n_k}{M}\right)\right]\right) \\
		& = (M^2\eta(0) - MT - MT + MT)^d = M^d (M\eta(0)-T)^d.
	\end{align*}
	Combining these we get
	$$
	\ETe(w_r) = -M^{d-2} (M\eta(0)-T)^d r^2,
	$$
	so that
	$$
	\exT(w_r) = \ETp(w_r) - \ETe(w_r) = M^{d-2}\left(T^d + (M\eta(0)-T)^d\right) r^2.
	$$
\end{proof}

For fixed $k \in [d] \coloneqq \{1, 2, \dots, d\}$ and coordinates $m_1, \dots, m_{k-1}, m_{k+1}, \dots, m_d \in \{0, 1, \dots, M-1\}$ the set
$$
R = \left\{\frac{1}{M}(m_1, \dots, m_d): m_k = 0, 1, \dots, M-1\right\} \subseteq G
$$
is called a \textbf{$k$-row} or simply \textbf{row}. For any $k$, the set of all $k$-rows forms a partition of $G$. For $\bx \in G$ and $k \in [d]$ we can thus define $R_k(\bx)$ to be the unique $k$-row containing $\bx$. Any given $\bx$ is contained in exactly $d$ rows, namely $R_1(\bx), \dots, R_d(\bx)$.

We will now define coefficients for a certain expression $g_R$ below. These coefficients will be such that we can get a relation as in Proposition \ref{thm:expression_for_excess} below. From the proof it will be clear how these coefficients have to be defined but without that foresight they will come out of the blue. We thus advice to skip the details for now and first look at the expression for $g_R$ (a linear functional in the variables $w(\bx)$ and $r$), then go into Proposition \ref{thm:expression_for_excess} to see why we are doing this. The form of the coefficients can then be extracted from the proof.

Define
\begin{align*} 
	\begin{split}
		\alpha^\bx_{R_k(\by)} = \sum_{\substack{P, Q \sbse [d] \\ k \notin Q}} &  \frac{\gamma(P, Q)}{2d-\#P-\#Q} \eta(0)^{d - \#(P \cup Q)} \\
		& \times \prod_{\ell \in P \sm Q} \eta(x_\ell) \prod_{\ell \in Q \sm P} \eta(y_\ell) \prod_{\ell \in P \cap Q} \etp(x_\ell, y_\ell),
	\end{split}
\end{align*}
and
\begin{align*} 
	\begin{split}
		\alpha^r_{R_k(\bx)} = \sum_{\substack{P, Q \sbse [d] \\ Q \neq [d] \\ k \notin P \sm Q}} & \frac{M^{d-1}}{d-\#(P \sm Q)} \frac{d-\#Q}{2d-\#P-\#Q} \left(\frac{T}{M}\right)^{\#Q} \gamma(P, Q) \\
		& \times \eta(0)^{d - \#(P \cup Q)} \prod_{\ell \in P \sm Q} \eta(x_\ell).
	\end{split}
\end{align*}
for $k \in [d]$ and $\bx, \by \in G$, where
\begin{align*}
	\gamma(P, Q) = \begin{cases}
		2 & , P=Q=\es \\
		0 & , P=Q=[d] \\
		0 & , P \neq \es, Q = \es \\
		0 & , P = \es, Q \neq \es \\
		-2(-1)^{\#P+\#Q} & , \text{else}
	\end{cases}
\end{align*}
and for any row $R$ of $G$ define the linear forms in the variables $w(\bx), \bx \in G$ and $r$ given by
$$
g^{}_R(w, r) \coloneqq \sum_{\bx \in G} \alpha^\bx_R w(\bx) + \alpha^r_R r.
$$
Note that $g_R(w, r)$ depends on all $w(\bx), \bx \in G$, not just on those points in $R$. Also, for $\alpha_R^\bx$ we choose a $\by$ such that $R = R_k(\by)$ and then use the above definition (and analogously for $\alpha_R^r$). It is important to emphasize that these yield well-defined quantities since $\alpha^\bx_{R_k(\by)}$ only depends on $x_1, \dots, x_d$ and $y_1, \dots, y_{k-1}, y_{k+1}, \dots, y_d$ (note that $R_k(\by) = R_k(\by')$ if and only if $\by$ and $\by'$ differ at most in their $k$-th coordinate) and analogously for $\alpha^r_{R_k(\bx)}$. We have the following general result.

\begin{prop} \label{thm:expression_for_excess}
	Let $M, d \in \bbN, d \geq 2$, $\cT$ be an energy triple and $r \in \bbR$. For $w: G \ra \bbR$ it holds
	\begin{align} \label{eq:excess}
		\exT(w) - \exT(w_r) = \sum_{R \text{ row}} g^{}_R(w, r) \left(\sum_{\bx \in R} w(\bx) - r\right),
	\end{align}
	where the sum runs over all rows of $G$ (and $w_r$ as in Lemma \ref{lem:excess_of_constant_weight}).
\end{prop}

We will be interested in weights such that every summand on the right-hand side of this expression is $0$ (see Theorem \ref{cor:Lp_Le_latin_hypercubes} below), which will show us that the excess of such a weight is constantly $\exT(w_r)$. For the proof we advice the following: consider the coefficients $\alpha_R^\bx$ and $\alpha_R^r$ as unknowns and see from the proof which relations they need to fulfil as to get the statement of the proposition. In this way, the expressions for the coefficients will turn up more naturally, taking the form of the expression from the proposition for granted. Indeed, case 1 in the proof will give the values for $\alpha_R^\bx$ (in particular equation \eqref{eq:gives_ax}) and case 3 will give $\alpha_R^r$ (in particular equation \eqref{eq:gives_ar}), while keeping in mind that $\alpha_{R_k(\by)}^\bx$ cannot depend on $y_k$ and $\alpha_{R_k(\bx)}^r$ cannot depend on $x_k$. The expression for $\gamma$ also follows from \eqref{eq:wx_wy_product} and \eqref{eq:gives_gamma} in case 1.

\begin{proof} [Proof of Proposition \ref{thm:expression_for_excess}]
	Note that both sides of \eqref{eq:excess} are homogeneous polynomials of degree $2$ in the variabls $w(\bx), \bx \in G$ and $r$. We thus show that both sides are equal by comparing coefficients. To do so we distinguish between four cases.
	
	\underline{Case 1, coefficient of $w(\bx)w(\by), \bx \neq \by$:} On the left-hand side only $\exT(w)$ contributes to coefficients of the form $w(\bx)w(\by)$. Concretely, on the left-hand side the coefficient is (looking at Definition \ref{def:F_energy})
	\begin{align} \label{eq:wx_wy_product}
		\begin{split}
			& 2 \prod_{\ell=1}^d \etp(x_\ell, y_\ell) + 2 \prod_{\ell=1}^d [\eta(0)-\eta(x_\ell)] + 2 \prod_{\ell=1}^d [\eta(0)-\eta(y_\ell)] \\
			& - 2 \prod_{\ell=1}^d [\eta(0) - \eta(x_\ell) - \eta(y_\ell) + \etp(x_\ell, y_\ell)]
		\end{split}
	\end{align}
	and on the right-hand side it is
	\begin{align} \label{eq:wx_wy_sum}
		\sum_{k=1}^d \alpha_{R_k(\bx)}^\by + \sum_{k=1}^d \alpha_{R_k(\by)}^\bx.
	\end{align}
	To show that these are the same we multiply out the products in \eqref{eq:wx_wy_product} to obtain
	$$
	\sum_{(A, B, C, D)} \Gamma_1(A, B, C, D) \eta(0)^{\#A} \prod_{\ell \in B} \eta(x_\ell) \prod_{\ell \in C} \eta(y_\ell) \prod_{\ell \in D} \etp(x_\ell, y_\ell)
	$$
	where the sum runs over all ordered partitions $(A, B, C, D)$ of $[d]$ (that is $[d] = A \dot{\cup} B \dot{\cup} C \dot{\cup} D$ disjointly with possibly empty blocks and where the order of the sets matters) and
	\begin{align*}
		\Gamma_1(A, B, C, D) = & 2 [D = [d]] + 2 (-1)^{\#B} [C = D = \es] + 2 (-1)^{\#C} [B = D = \es] \\
		& - 2 (-1)^{\#B + \#C} \\
		= & \begin{cases}
			2 & , A = [d] \\
			0 & , D = [d] \\
			0 & , A \neq [d], C=D=\es \\
			0 & , A \neq [d], B=D=\es \\
			-2(-1)^{\#B + \#C} & , \text{else}
		\end{cases},
	\end{align*}
	with the Iverson bracket $[\cdot]$, that is $[S] = 1$ if the statement $S$ in the bracket is true, otherwise $[S] = 0$ if $S$ is false. If we show that the coefficient of the term
	$$
	\eta(0)^{\#A} \prod_{\ell \in B} \eta(x_\ell) \prod_{\ell \in C} \eta(y_\ell) \prod_{\ell \in D} \etp(x_\ell, y_\ell)
	$$
	in \eqref{eq:wx_wy_sum} is precisely $\Gamma_1(A, B, C, D)$ for any ordered partition $(A, B, C, D)$ we are done with this case. Observe the bijection
	\begin{align*}
		\{(P, Q): P, Q \sbse [d]\} & \ra \{(A, B, C, D) \text{ ordered set partition of } [d]\} \\
		(P, Q) & \mapsto ([d] \sm (P \cup Q), P \sm Q, Q \sm P, P \cap Q)
	\end{align*}
	with inverse given by
	$$
	(A, B, C, D) \mapsto (B \cup D, C \cup D).
	$$
	Under this bijection we also have
	\begin{align} \label{eq:gives_gamma}
		\Gamma_1(A, B, C, D) = \gamma(B \cup D, C \cup D)
	\end{align}
	for all ordered partitions $(A, B, C, D)$. We thus need to show
	\begin{align} \label{eq:gives_ax}
		\begin{split}
			\sum_{k=1}^d \alpha_{R_k(\bx)}^\by + \sum_{k=1}^d \alpha_{R_k(\by)}^\bx = \sum_{P, Q \sbse [d]} & \gamma(P, Q) \eta(0)^{d - \#(P \cup Q)} \\
			& \times \prod_{\ell \in P \sm Q} \eta(x_\ell) \prod_{\ell \in Q \sm P} \eta(y_\ell) \prod_{\ell \in P \cap Q} \etp(x_\ell, y_\ell).
		\end{split}
	\end{align}
	Fix $(P, Q)$ and we will distribute the term
	$$
	\gamma(P, Q) \eta(0)^{d - \#(P \cup Q)} \prod_{\ell \in P \sm Q} \eta(x_\ell) \prod_{\ell \in Q \sm P} \eta(y_\ell) \prod_{\ell \in P \cap Q} \etp(x_\ell, y_\ell)
	$$
	onto the coefficients $\alpha_{R_k(\bx)}^\by$ and $\alpha_{R_k(\by)}^\bx$ whenever possible (that is when $x_k$ or $y_k$ does not show up in this product). Since $\gamma([d], [d]) = 0$ we may ignore the term with the expression for $(P, Q) = ([d], [d])$, so that for every such expression there is a variable $x_k$ or $y_k$ which does not show up in it. In total, this product will get included in $\alpha_{R_k(\bx)}^\by$ for all $k \notin P$, and it will get included in $\alpha_{R_k(\by)}^\bx$ for all $k \notin Q$. Thus, the product will get distributed onto
	$$
	(d - \#P) + (d - \# Q) = 2d-\#P-\#Q
	$$
	coefficients, meaning that we need to divide by this amount to normalize it appropriately. This gives the expressions for $\alpha_{R_k(\by)}^\bx$ and finishes this case.
	
	\underline{Case 2, coefficient of $w(\bx)^2, \bx \in G$:} Similar to the first case, we now have to show the equality between
	\begin{align*} 
		\begin{split}
			& \prod_{\ell=1}^d \etp(x_\ell, x_\ell) + 2 \prod_{\ell=1}^d [\eta(0)-\eta(x_\ell)]  - \prod_{\ell=1}^d [\eta(0) - 2 \eta(x_\ell) + \etp(x_\ell, x_\ell)]
		\end{split}
	\end{align*}
	and
	\begin{align*} 
		\sum_{k=1}^d \alpha_{R_k(\bx)}^\bx.
	\end{align*}
	The equality of these two expressions follows by case 1 where we set $\by = \bx$.

	\underline{Case 3, coefficient of $w(\bx) r, \bx \in G$:} Neither $\exT(w)$ nor $\exT(w_r)$ contain any terms of the form $w(\bx)r$, giving a $0$ on the left-hand side of \eqref{eq:excess}. From the right-hand side we get
	$$
	\sum_{k=1}^d \alpha_{R_k(\bx)}^r - \sum_{R \text{ row}} \alpha_R^\bx
	$$
	and thus it remains to verify
	\begin{align} \label{eq:wx_r}
		\sum_{k=1}^d \alpha_{R_k(\bx)}^r = \sum_{R \text{ row}} \alpha_R^\bx
	\end{align}
	for all $x \in G$. For the right-hand side note that we can rewrite the sum as
	\begin{align*} 
		\sum_{R \text{ row}} \alpha_R^\bx = \frac{1}{M} \sum_{y \in G} \sum_{k = 1}^d \alpha_{R_k(\by)}^\bx
	\end{align*}
	just by noting that in the double sum every row is summed over $M$-times. We can also rewrite the following sum as
	\begin{align*}
		\sum_{k = 1}^d \alpha_{R_k(\by)}^\bx = & \sum_{k=1}^d \sum_{\substack{P, Q \sbse [d] \\ k \notin Q}} \frac{\gamma(P, Q)}{2d-\#P-\#Q} \eta(0)^{d - \#(P \cup Q)} \\
		& \times \prod_{\ell \in P \sm Q} \eta(x_\ell) \prod_{\ell \in Q \sm P} \eta(y_\ell) \prod_{\ell \in P \cap Q} \etp(x_\ell, y_\ell) \\
		= & \sum_{\substack{P, Q \sbse [d] \\ Q \neq [d]}} \frac{(d - \#Q)\gamma(P, Q)}{2d-\#P-\#Q} \eta(0)^{d - \#(P \cup Q)} \\
		& \times \prod_{\ell \in P \sm Q} \eta(x_\ell) \prod_{\ell \in Q \sm P} \eta(y_\ell) \prod_{\ell \in P \cap Q} \etp(x_\ell, y_\ell).
	\end{align*}
	Noting now
	\begin{align*}
		& \sum_{y \in G} \prod_{\ell \in [d] \sm Q} 1 \prod_{\ell \in Q \sm P} \eta(y_\ell) \prod_{\ell \in P \cap Q} \etp(x_\ell, y_\ell) \\
		= & \prod_{\ell \in [d] \sm Q} \left(\sum_{m=0}^{M-1} 1\right) \prod_{\ell \in Q \sm P} \left(\sum_{m=0}^{M-1} \eta\left(\frac{m}{M}\right)\right) \prod_{\ell \in P \cap Q} \left(\sum_{m=0}^{M-1} \etp\left(x_\ell, \frac{m}{M}\right)\right) \\
		= & M^{d - \#Q} T^{\#Q}
	\end{align*}
	we get
	\begin{align*}
		\sum_{\by \in G} \sum_{k = 1}^d \alpha_{R_k(\by)}^\bx = & \sum_{\substack{P, Q \sbse [d] \\ Q \neq [d]}} \frac{(d - \#Q)\gamma(P, Q)}{2d-\#P-\#Q} \eta(0)^{d - \#(P \cup Q)} \\
		& \times \prod_{\ell \in P \sm Q} \eta(x_\ell) \sum_{y \in G} \prod_{\ell \in Q \sm P} \eta(y_\ell) \prod_{\ell \in P \cap Q} \etp(x_\ell, y_\ell) \\
		= & M^d \sum_{\substack{P, Q \sbse [d] \\ Q \neq [d]}} \frac{d-\#Q}{2d-\#P-\#Q} \left(\frac{T}{M}\right)^{\#Q} \gamma(P, Q) \\
		& \times \eta(0)^{d - \#(P \cup Q)} \prod_{\ell \in P \sm Q} \eta(x_\ell).
	\end{align*}
	It remains to verify
	\begin{align} \label{eq:gives_ar}
		\begin{split}
			\sum_{k=1}^d \alpha_{R_k(\bx)}^r = M^{d-1} \sum_{\substack{P, Q \sbse [d] \\ Q \neq [d]}} & \frac{d-\#Q}{2d-\#P-\#Q} \left(\frac{T}{M}\right)^{\#Q} \gamma(P, Q) \\
			& \times \eta(0)^{d - \#(P \cup Q)} \prod_{\ell \in P \sm Q} \eta(x_\ell).
		\end{split}
	\end{align}
	Note that the sum on the right-hand side can be written as (using $\gamma([d], \es) = 0$, so discarding the summand for $(P, Q) = ([d], \es)$)
	\begin{align*}
		\sum_{k=1}^d \sum_{\substack{P, Q \sbse [d] \\ Q \neq [d] \\ k \notin P \sm Q}} \frac{M^{d-1}}{d-\#(P \sm Q)} & \frac{d-\#Q}{2d-\#P-\#Q} \left(\frac{T}{M}\right)^{\#Q} \gamma(P, Q) \\
		& \times \eta(0)^{d - \#(P \cup Q)} \prod_{\ell \in P \sm Q} \eta(x_\ell).
	\end{align*}
	But then \eqref{eq:wx_r} follows simply by the definition of $\alpha_{R_k(\bx)}^r$ (again, as in case 1, we can see which terms can get included in a given coefficient $\alpha_{R_k(\bx)}^r$ and distribute them onto all possible such coefficients and dividing at the end by the corresponding amount to normalize).
	
	\underline{Case 4, coefficient of $r^2$:} For the last case, the left-hand side of \eqref{eq:excess} yields, by Lemma \ref{lem:excess_of_constant_weight}, a coefficient of
	$$
	-M^{d-2}\left(T^d + (M\eta(0)-T)^d\right)
	$$
	for $r^2$. On the right-hand side we get
	$$
	-\sum_{R \text{ row}} \alpha_R^r
	$$
	and it remains to show
	\begin{align} \label{eq:r_r}
		M^{d-2}\left(T^d + (M\eta(0)-T)^d\right) = \sum_{R \text{ row}} \alpha_R^r.
	\end{align}
	For this we use what we have already proven in the previous cases. By case 3 we can write
	$$
	\sum_{R \text{ row}} \alpha_R^r = \frac{1}{M} \sum_{\bx \in G} \sum_{k = 1}^d \alpha_{R_k(\bx)}^r = \frac{1}{M} \sum_{\bx \in G} \sum_{R \text{ row}} \alpha_R^\bx = \frac{1}{M^2} \sum_{\bx \in G} \sum_{\by \in G} \sum_{k = 1}^d \alpha_{R_k(\by)}^\bx
	$$
	which gives us
	$$
	2 \sum_{R \text{ row}} \alpha_R^r = \frac{1}{M^2} \sum_{\bx \in G} \sum_{\by \in G} \left(\sum_{k = 1}^d \alpha_{R_k(\by)}^\bx + \sum_{k = 1}^d \alpha_{R_k(\bx)}^\by\right).
	$$
	By cases 1 and 2 we can write, no matter if $\bx \neq \by$ or $\bx = \by$,
	\begin{align*}
		& \sum_{k = 1}^d \alpha_{R_k(\by)}^\bx + \sum_{k = 1}^d \alpha_{R_k(\bx)}^\by \\
		= & 2 \prod_{\ell=1}^d \etp(x_\ell, y_\ell) + 2 \prod_{\ell=1}^d [\eta(0)-\eta(x_\ell)] + 2 \prod_{\ell=1}^d [\eta(0)-\eta(y_\ell)] \\
		& - 2 \prod_{\ell=1}^d [\eta(0) - \eta(x_\ell) - \eta(y_\ell) + \etp(x_\ell, y_\ell)].
	\end{align*}
	Now \eqref{eq:r_r} follows from the identities
	$$
	\sum_{\bx \in G} \sum_{\by \in G} \prod_{\ell=1}^d \etp(x_\ell, y_\ell) = M^d T^d,
	$$
	$$
	\sum_{\bx \in G} \sum_{\by \in G} \prod_{\ell=1}^d [\eta(0)-\eta(x_\ell)] = \sum_{\bx \in G} \sum_{\by \in G} \prod_{\ell=1}^d [\eta(0)-\eta(y_\ell)] = M^d (M\eta(0)-T)^d,
	$$
	$$
	\sum_{\bx \in G} \sum_{\by \in G} \prod_{\ell=1}^d [\eta(0) - \eta(x_\ell) - \eta(y_\ell) + \etp(x_\ell, y_\ell)] = M^d (M\eta(0)-T)^d,
	$$
	which we have already shown in the proof of Lemma \ref{lem:excess_of_constant_weight}. This finishes this case and also the proof of Proposition \ref{thm:expression_for_excess}.
\end{proof}

The above proposition might not seem like more than a curious algebraic identity at first, but it has some interesting implications for the extreme and periodic discrepancy of certain point sets, to be discussed in the next section. In essence, it is a more general and abstract version of Theorem \ref{thm:Lp_Le_permutation}.

\section{Weak Latin hypercubes}

Latin hypercubes are extensions of permutation sets as introduced in Definition \ref{def:permutation_set}, that have been studied for quite some time now (see for example \cite{GH21, HL98, Loh96, Mat00, PM05, Ste87} for a small selection). We will need an adapted version which we choose to call \emph{weak Latin hypercubes}. We start by introducing this notion.

A \textbf{weak $M$-Latin hypercube} (or just \textbf{weak Latin hypercube}) of dimension $d$ is a subset $\cH \sbse \{0, 1, \dots, M-1\}^d$ such that for any $k \in [d]$ and any $m_1, \dots, m_{k-1}, m_{k+1}, \dots, m_d \in \{0, 1, \dots, M-1\}$ there is exactly one $m_k \in \{0, 1, \dots, M-1\}$ with $(m_1, \dots, m_d) \in \cH$. That is to say, $\cH$ contains exactly one element of each row of $G$. Every weak $M$-Latin hypercube of $G$ generates a point set $X(\cH) \coloneqq \frac{1}{M} \cH$ ($\cH$ scaled into the unit cube $[0, 1)^d$) of size $\# X(\cH) = M^{d-1}$. For $d=2$ we recover permutation sets as in Definition \ref{def:permutation_set}.

In contrast, a Latin hypercube (as in the literature above) is a set $X \sbse G$ such that every \emph{hyperplane} given by $\{\bx \in [0,1 )^d: x_k = i/N\}$ for some $k \in [d], i \in \{0, 1, \dots, N-1\}$ contains exactly one element from $X$. We want to emphasize that weak Latin hypercubes really are a different class of point sets (for $d \geq 3$). What we call weak Latin hypercube in this paper follows more the combinatorial definition of a Latin hypercube given in \cite{MW08}.

There is a canoncial bijection between weights $w: G \ra \{0, 1\}$ having row sums all equal to $1$ and point sets $X$ constructed from weak Latin hypercubes in the above way, namely
$$
w \mapsto X = \{\bx \in G: w(\bx) = 1\}.
$$
We can thus apply Proposition \ref{thm:expression_for_excess} to conclude the following.

\begin{thm} \label{cor:Lp_Le_latin_hypercubes}
	If $X \sbse [0, 1)^d$ is constructed from a weak $M$-Latin hypercube then
	$$
	\Lp(X)^2 - 2^d \Le(X)^2 = \frac{(2M^2+1)^d + (M^2-1)^d - (1+2^d)M^{2d}}{6^d M^2}.
	$$
\end{thm}

\begin{proof}
	Let $X \sbse [0, 1)^d$ be constructed from a weak $M$-latin hypercube. Consider the weight $w: G \ra \{0, 1\}$ given by $w(\bx) = 1$ if and only if $\bx \in X$ (that is the characteristic function of $X$) and note that, as discussed above, $w$ has constant row sum equal to $1$. Applying Proposition \ref{thm:expression_for_excess} with $r=1$, noting that $\sum_{\bx \in R} w(x) - 1 = 0$ for all rows $R$ and using Lemma \ref{lem:excess_of_constant_weight} we get
	$$
	\exT(w) = \exT(w_1) = M^{d-2}\left(T^d + (M \eta(0) - T)^d\right).
	$$
	To recover the $L_2$-discrepancy from this we choose $\eta(s) = \frac{1}{2} - s + s^2$ and the corresponding energy triple $\cT$ as in Example \ref{ex:discrepancy_as_energy}. It holds that $\eta(0) = 1/2$ and
	\begin{align} \label{eq:Lp_energy_sum}
		T = \sum_{m=0}^{M-1} \frac{1}{2} - \frac{m}{M} + \frac{m^2}{M^2} = \frac{2M^2+1}{6M}.
	\end{align}
	Again by Example \ref{ex:discrepancy_as_energy}, relating $\ETe$ and $\ETp$ to $\Le$ and $\Lp$, we see that
	\begin{align*}
		& \Lp(X)^2 - 2^d \Le(X)^2 = \exT(w) - \left(\sum_{\bx \in G} w(\bx)\right)^2 \left(\frac{1}{3^d} + \frac{1}{6^d}\right) \\
		= & M^{d-2}\left[\left(\frac{2M^2+1}{6M}\right)^d + \left(\frac{M}{2} - \frac{2M^2+1}{6M}\right)^d\right] - M^{2(d-1)} \frac{2^d+1}{6^d},
	\end{align*}
	which simplifies to the desired result.
\end{proof}

\begin{rem}
	The proof even shows that $\ETp(w) - \ETe(w)$ does not depend on $w$ as long as $w$ has constant row sums all equal to some $r \in \bbR$. This is a quadratic form in the variables $w(\bx), \bx \in G$, for which it can be verified separately that all non-constant coefficients vanish. In this case Proposition \ref{thm:expression_for_excess} actually gives even more information (also including the case when $w$ does not have constant row sums). For the sake of generality, we chose to go this more involved route.
\end{rem}

For $d=2$ we recover Theorem \ref{thm:Lp_Le_permutation}. This can be generalized a bit with only slight adjustments to the proof (just replace the occurences of $w_1$ with $w_r$ in the above proof) to the following.

\begin{thm} \label{cor:Lp_Le_union_latin_hypercubes}
	If $X \sbse [0, 1)^d$ is constructed from the union of $r$ disjoint weak $M$-Latin hypercubes (in the obvious way) then
	$$
	\Lp(X)^2 - 2^d \Le(X)^2 = r^2 \frac{(2M^2+1)^d + (M^2-1)^d - (1+2^d)M^{2d}}{6^d M^2}.
	$$
\end{thm}

A special case of this would be $X = G$, the entire discretized torus itself for $r = M$ (compare with Remark 12 in \cite{HKP20}).

Let $X \sbse [0, 1)^d$ be constructed from a weak $M$-Latin hypercube. Since $\Le(X)^2 \geq 0$ we see by Theorem \ref{cor:Lp_Le_latin_hypercubes} that
\begin{align} \label{ineq:Lp_first_lower_bound}
	\Lp(X)^2 \geq \frac{d(2^{d-1}-1)}{6^d} M^{2d-4} + O (M^{2d-6}),
\end{align}
in terms of $N \coloneqq \#X = M^{d-1}$ we get $\Lp(X) \gtrsim N^{\frac{d-2}{d-1}}$. It is natural to ask if a similar lower bound holds for $\Le(X)$. This will turn out to be the case. To show this, it actually suffices to improve on the coefficient of $M^{2d-4}$ in \eqref{ineq:Lp_first_lower_bound}.

For this, note that $\ETp(w)$ can be seen as a quadratic form in $w \in \bbR^G$, i.e. we can write
$$
\ETp(w) = w^\top E w
$$
with $E \in \bbR^{G \times G}$ and matrix entries $E_{\bx \by} = \prod_{k=1}^d \etp(x_k, y_k), \bx, \by \in G$. It turns out that in the canonical case (of the form as in Proposition \ref{prop:energy_triple_canonical_construction}),  $E$ is actually a very structured matrix which allows us to find its eigenvalue decomposition. For this define $\Gbar \coloneqq \{0, 1, \dots, N-1\}^d$. The following can be seen as an eigenvalue decomposition of the $d$-fold Kronecker product of real symmetric, circulant matrices and the proof is very similar to the case $d=1$ of circulant matrices (see for example \cite{Dav79, KSU03, KS12, Vyb11}). In essence it is the well known result that the multidimensional discrete Fourier transform of a convolution becomes a componentwise multiplication in the frequency domain.

\begin{thm} \label{thm:svd_energy}
	Let $\eta: [0, 1) \ra \bbR$ with $\eta(s) = \eta(1-s)$ for all $s \in (0, 1)$ and let $\cT$ be the canonical energy triple via $\eta$. For $w: G \ra \bbR$ its periodic $\cT$-energy fulfils
	\begin{align*}
		\sum_{\bx, \by \in G} w(x) w(y) \prod_{k=1}^d \eta(|x_k - y_k|) = M^{-d} \sum_{\boldf \in \Gbar} \lambda_\boldf \left|\sum_{\bx \in G} w(x) \exp(2\pi\iu \boldf \cdot \bx)\right|^2,
	\end{align*}
	where $\iu^2 = -1$, $\boldf \cdot \bx$ is the usual scalar product and
	$$
	\lambda_\boldf = \prod_{k=1}^d \left(\sum_{m=0}^{N-1} \eta\left(\frac{m}{M}\right) \exp\left(2\pi\iu f_k \frac{m}{M}\right)\right).
	$$
\end{thm}

Since $\eta(s) = \eta(1-s)$ we could even write
$$
\lambda_\boldf = \prod_{k=1}^d \left(\sum_{m=0}^{M-1} \eta\left(\frac{m}{M}\right) \cos\left(2\pi f_k \frac{m}{M}\right)\right).
$$

The following could be seen as a discrete version of Proposition 3 of \cite{HO16} (equation \eqref{eq:periodisc_disc_Fourier} above) and is also similar to Lemma 2.8 of \cite{KN12}. Also note the similarities to diaphony (\cite{Lev95, Zin76}, Definition 1.29 in \cite{DT06} and \eqref{eq:diaphony} above)

\begin{cor} \label{cor:svd_periodic discrepancy}
	Let $w: G \ra \bbR$, then
	$$
	\Lp(G, w)^2 = -\frac1{3^d} \left(\sum_{\bx \in G} w(\bx)\right)^2 + \sum_{\boldf \in \Gbar} \mu_\boldf \left|\sum_{\bx \in G} w(\bx) \exp(2\pi\iu \boldf \cdot \bx)\right|^2
	$$
	with
	$$
	\mu_\boldf = \prod_{k=1}^d \begin{cases}
		\frac13 + \frac{1}{6M^2} & , f_k = 0 \\
		\frac{1}{2M^2 \sin^2(\pi f_k/M)} & , f_k \neq 0
	\end{cases}
	$$
	for $\boldf \in \Gbar$.
\end{cor}

Note the similarities to \eqref{eq:diaphony} and \eqref{eq:periodisc_disc_Fourier}.

\begin{proof}
	By Theorem \ref{thm:svd_energy} it remains to calculate $\lambda_\boldf$ for the case $\eta(s) = \frac{1}{2}-s+s^2$. For this we need the sum
	$$
	\sum_{m=0}^{M-1} \left(\frac{1}{2} - \frac{m}{M} + \frac{m^2}{M^2}\right) \exp\left(2\pi\iu f_k \frac{m}{M}\right).
	$$
	The case $f_k = 0$ reduces to \eqref{eq:Lp_energy_sum} so we only need $f_k \neq 0$. We use the elementary sums
	\begin{align*}
		\sum_{m=0}^{M-1} & e^{\iu mt} = \frac{1-e^{\iu Mt}}{1-e^{\iu t}}, \\
		\sum_{m=0}^{M-1} & m e^{\iu mt} = \frac{e^{\iu t} - M e^{\iu Mt} + (M-1)e^{\iu (M+1)t}}{\left(1-e^{\iu t}\right)^2}, \\
		\sum_{m=0}^{M-1} & m^2 e^{\iu mt} \\
		& = \frac{e^{\iu t} + e^{2\iu t} - M^2 e^{\iu Mt} + (2M^2-2M-1)e^{\iu (M+1)t} - (M-1)^2e^{\iu (M+2)t}}{\left(1-e^{\iu t}\right)^3}
	\end{align*}
	for $t \in \bbR$. Setting $t = 2\pi \frac{f_k}{M}$, using the $2\pi$-periodicity of $t \mapsto e^{\iu t}$ and simplifying accordingly gives
	\begin{align*}
		& \sum_{m=0}^{M-1} \left(\frac{1}{2} - \frac{m}{M} + \frac{m^2}{M^2}\right) \exp\left(2\pi\iu f_k \frac{m}{M}\right) \\
		= & -\frac{2}{M} \cdot \frac{\exp\left(2\pi\iu \frac{f_k}{M}\right)}{\left(1-\exp\left(2\pi\iu \frac{f_k}{M}\right)\right)^2} = \frac{1}{M \left(1 - \cos\left(2\pi \frac{f_k}{M}\right)\right)},
	\end{align*}
	where we can additionally write $1-\cos(2t) = 2 \sin^2(t)$. Then $\mu_\boldf = M^{-d} \lambda_\boldf$ and the statement follows.
\end{proof}

\begin{rem} \label{rem:svd_periodic_discrepancy} \
	\begin{itemize}
		\item [(i)] The factors in the definition of $\lambda_\boldf$ are a discrete Fourier transform of the function $\eta(s) = \frac{1}{2}-s+s^2$ on $[0, 1)$. Applying the inverse transform gives the identity
		\begin{align*} 
			\frac{1}{2} - \frac{m}{M} + \frac{m^2}{M^2} = \frac13 + \frac{1}{6M^2} + \frac{1}{2M^2} \sum_{n=1}^{M-1} \frac{\cos\left(2\pi m \frac{n}{M}\right)}{\sin^2\left(\pi \frac{n}{M}\right)}
		\end{align*}
		for $m = 0, 1, \dots, M-1$.
		
		\item [(ii)] For a permutation set $X(\sigma)$ we can write the formula of Theorem \ref{thm:svd_energy}, writing $\boldf = (f_1, f_2)$ and noticing that $\lambda_{(0, 0)} = T^2$ (see Definition \ref{def:energy_triple} (i)), as
		\begin{align*} 
			\ETp(\sigma) = \frac{T^2}{M^2} + \frac{1}{M^2} \sum_{f_1, f_2 = 1}^{M-1} \lambda_\boldf \left|\sum_{m = 0}^{M-1} \exp\left(\frac{2\pi\iu}{M}[f_1 m + f_2 \sigma(m)]\right)\right|^2,
		\end{align*}
		where terms corresponding to $\boldf = (f_1, 0)$ and $(0, f_2)$ vanish (except for $\boldf = (0, 0)$), and in the case of Corollary \ref{cor:svd_periodic discrepancy} as
		\begin{align*} 
			\begin{split}
				\Lp(\sigma)^2 = \frac{1}{9} + \frac{1}{36N^2} + \frac{1}{4N^4} \sum_{f_1, f_2 = 1}^{N-1} & \frac{1}{\sin^2(\pi f_1/N) \sin^2(\pi f_2/N)} \\
				\times & \left|\sum_{m=0}^{N-1}\exp\left(\frac{2\pi\iu}{N}[f_1 m + f_2 \sigma(m)]\right)\right|^2.
			\end{split}
		\end{align*}
		This is a generalization of the formula for the periodic $L_2$-discrepancy of rational lattices in Theorem 10 of \cite{HKP20}.
	\end{itemize}
\end{rem}


We come back to the lower bound for the periodic and extreme $L_2$-discrepancy of weak Latin hypercubes as discussed under Theorem \ref{cor:Lp_Le_union_latin_hypercubes}. In particular, we improve a bit on the bound in \eqref{ineq:Lp_first_lower_bound}.

\begin{thm} \label{cor:lower_bound_latin_hypercube_discrepancy}
	Let $X \sbse [0, 1)^d, d \geq 2$ be constructed from a weak $M$-Latin hypercube and $N = \#X = M^{d-1}$. Then
	$$
	\Lp(X) \geq \left(\frac{d}{2 \cdot 3^d}\right)^{1/2} N^{\frac{d-2}{d-1}}
	$$
	and
	$$
	\Le(X) \geq \left(\frac{d}{12^d}\right)^{1/2} (1-o(1)) \cdot N^{\frac{d-2}{d-1}},
	$$
	where $o(1) \longrightarrow 0$ as $N \longrightarrow \infty$.
\end{thm}

\begin{proof}
	Notice that all the $\lambda_\boldf$ in Corollary \ref{cor:svd_periodic discrepancy} are positive. Thus, using only the summand corresponding to $\boldf = \mathbf{0}$ we get
	$$
	\sum_{\boldf \in \Gbar} \mu_\boldf \left|\sum_{\bx \in G} w(x) \exp(2\pi\iu \boldf \cdot \bx)\right|^2 \geq \left(\frac13 + \frac{1}{6M^2}\right)^d \left(\sum_{\bx \in G} w(\bx)\right)^2.
	$$
	Using $\sum_{\bx \in G} w(\bx) = \#X = M^{d-1}$ we get
	$$
	\Lp(X)^2 \geq -\frac1{3^d} M^{2(d-1)} + \left(\frac{1}{3} + \frac{1}{6M^2}\right)^d M^{2(d-1)} \geq \frac{d}{2 \cdot 3^d} M^{2d-4} = \frac{d}{2 \cdot 3^d} N^{2 \frac{d-2}{d-1}},
	$$
	proving the bound on $\Lp(X)$. As for $\Le(X)$, we employ Corollary \ref{cor:Lp_Le_latin_hypercubes} and the just proven bound for $\Lp(X)$ to get
	\begin{align*}
		2^d \Le(X)^2 & = \Lp(X)^2 - \frac{(2M^2+1)^d + (M^2-1)^d - (1+2^d)M^{2d}}{6^d M^2} \\
		& \geq \frac{d}{2 \cdot 3^d} M^{2d-4} - \frac{d}{2}\left(\frac{1}{3^d} - \frac{2}{6^d}\right) M^{2d-4} - O(M^{2d-6}) \\
		& = \frac{d}{6^d} (1-o(1)) M^{2d-4} = \frac{d}{6^d} (1-o(1)) N^{2 \frac{d-2}{d-1}},
	\end{align*}
	from which the bound on $\Le(X)$ follows.
\end{proof}

The question remains how good these bounds are. One way to analyse this is through random point sets. This will be discussed in the next section.

\section{Random weak Latin hypercubes}

Consider randomly constructed weak Latin hypercubes. By that we mean that, for given parameters $M$ and $d$, among all $d$-dimensional weak $M$-Latin hypercubes we choose one $\cH$ uniformly at random and consider its periodic $L_2$-discrepancy $\Lp(\cH)^2$. We will determine the expected squared periodic $L_2$-discrepancy as $\cH$ runs uniformly through all weak Latin hypercubes. To do so we need some preparation.

Note that we can interprete a weak Latin hypercube $\cH$ as a function $\cH: \{0, 1, \dots, M-1\}^{d-1} \ra \{0, 1, \dots, M-1\}$ where $\cH(m_1, \dots, m_{d-1}) \coloneqq m_d$ is determined by the unique point $(m_1, \dots, m_d)$ in the latin hypercube with first $d-1$ coordinates determined by $m_1, \dots, m_{d-1}$. Denote by $\Lambda = \Lambda_M^d$ the number of $d$-dimensional weak $M$-latin hypercubes (see \cite{MW08} for more on these quantities). For any given $\bm, \bn \in \{0, 1, \dots, M-1\}^{d-1}$ we will need the probability distribution
$$
\prob_{\bm, \bn}(p, q) \coloneqq \Lambda^{-1} \cdot \#\{\cH: (\cH(\bm), \cH(\bn)) = (p, q)\}.
$$

\begin{lem} \label{lem:probability_distribution}
	Given $\bm, \bn \in \{0, 1, \dots, M-1\}^{d-1}$ set $\delta \coloneqq \#\{i=1, \dots, d-1: m_i \neq n_i\}$ (the \emph{Hamming distance}). Then
	$$
	\prob_{\bm, \bn}(p, q) = \prob_\delta(p, q) = \begin{cases}
		\frac{(M-1)^{\delta-1} - (-1)^{\delta-1}}{M^2 (M-1)^{\delta-1}} & , p = q \\
		\frac{(M-1)^\delta - (-1)^\delta}{M^2 (M-1)^\delta} & , p \neq q
	\end{cases}.
	$$
\end{lem}

\begin{proof}
	By symmetry, without loss of generality we may assume $m_1 = \dots = m_{d-1} = 0$ and $n_1 = \dots = n_\delta = 1, n_{\delta+1} = \dots = n_{d-1} = 0$. To further ease notation, it suffices to consider $\prob_{\bm, \bn}(0, 0)$ and $\prob_{\bm, \bn}(0, 1)$. It is easy to see that
	$$
	\bbP(\cH(\bm) = 0) = \frac{1}{M}.
	$$
	As for $\cH(\bn)$, by inclusion-exclusion (counting how often $\cH(\bk) = 0$ or $1$ appears in slices of the type $\bk \in \{0, 1, \dots, M-1\}^\beta \times \{0\}^{d-1-\beta}, \beta = 0, 1, \dots, \delta$ and coordinate permutations thereof while staying in $\{0, 1, \dots, M-1\}^\delta \times \{0\}^{d-1-\delta}$) we see that for all weak Latin hypercubes $\cH$ with $\cH(\mathbf{0}) = 0$ it holds
	\begin{align*}
		& \#\{\bk \in \{1, \dots, M-1\}^\delta \times \{0\}^{d-1-\delta}: \cH(\bk) = 0\} \\
		= & M^{\delta-1} - {\delta \choose 1} M^{\delta-2} + \dots + (-1)^{\delta - 1} {\delta \choose \delta-1} M^0 + (-1)^\delta \\
		= & \frac{1}{M}\left[(M-1)^\delta - (-1)^\delta\right] + (-1)^\delta
	\end{align*}
	and
	\begin{align*}
		& \#\{\bk \in \{1, \dots, M-1\}^\delta \times \{0\}^{d-1-\delta}: \cH(\bk) = 1\} \\
		= & M^{\delta-1} - {\delta \choose 1} M^{\delta-2} + \dots + (-1)^{\delta - 1} {\delta \choose \delta-1} M^0 \\
		= & \frac{1}{M}\left[(M-1)^\delta - (-1)^\delta\right].
	\end{align*}
	Again by symmetry we conclude, since $\# \left(\{1, \dots, M-1\}^\delta \times \{0\}^{d-1-\delta}\right) = (M-1)^\delta$,
	$$
	\prob_{\bm, \bn}(0, 0) = \frac{1}{M} \cdot \frac{\frac{1}{M}\left[(M-1)^\delta - (-1)^\delta\right] + (-1)^\delta}{(M-1)^\delta}
	$$
	and
	$$
	\prob_{\bm, \bn}(0, 1) = \frac{1}{M} \cdot \frac{\frac{1}{M}\left[(M-1)^\delta - (-1)^\delta\right]}{(M-1)^\delta},
	$$
	giving the claim.
\end{proof}

\begin{thm} \label{thm:expectation_energy}
	Given $d, M \in \bbN, d \geq 2, M \geq 2$ and an energy triple $\cT$, set
	$$
	\Delta \coloneqq \sum_{m=0}^{M-1} \etp\left(\frac{m}{M}, \frac{m}{M}\right).
	$$
	Then, as $\cH$ runs uniformly through all $d$-dimensional weak $M$-Latin hypercubes we have
	$$
	\bbE \ETp(\cH) = M^{d-2} \left(\frac{(\Delta-T)^d}{(M-1)^{d-1}} + T^d\right)
	$$
\end{thm}

\begin{proof}
	We want to calculate
	$$
	\Lambda^{-1} \sum_{\cH} \sum_{m_1, n_1=0}^{M-1} \dots \sum_{m_{d-1}, n_{d-1}=0}^{M-1} \prod_{k=1}^{d-1} \etp\left(\frac{m_k}{M}, \frac{n_k}{M}\right) \cdot \etp\left(\frac{\cH(\bm)}{M}, \frac{\cH(\bn)}{M}\right).
	$$
	Pulling the expectation inside the sum and using Lemma \ref{lem:probability_distribution} we get
	\begin{align*}
		\bbE \ETp(\cH) = & \sum_{m_1, n_1=0}^{M-1} \dots \sum_{m_{d-1}, n_{d-1}=0}^{M-1} \prod_{k=1}^{d-1} \etp\left(\frac{m_k}{M}, \frac{n_k}{M}\right) \\
		& \times \sum_{p, q = 0}^{M-1} \prob_{\delta}(p, q) \etp\left(\frac{p}{M}, \frac{q}{M}\right)
	\end{align*}
	with $\delta = \delta(\bm, \bn)$ as in Lemma \ref{lem:probability_distribution}. Writing
	\begin{align*}
		& \sum_{p, q = 0}^{M-1} \prob_{\delta}(p, q) \etp\left(\frac{p}{M}, \frac{q}{M}\right) \\
		= & \frac{1}{M(M-1)^\delta} \left((-1)^\delta \sum_{p=0}^{M-1} \etp\left(\frac{p}{M}, \frac{p}{M}\right) + \frac{(M-1)^\delta - (-1)^\delta}{M} \sum_{p, q=0}^{M-1} \etp\left(\frac{p}{M}, \frac{q}{M}\right)\right) \\
		= & \frac{(-1)^\delta \Delta + \left[(M-1)^\delta - (-1)^\delta\right] T}{M(M-1)^\delta} = \left(-\frac{1}{M-1}\right)^\delta \frac{\Delta - T}{M} + \frac{T}{M}.
	\end{align*}
	The above expression can thus be factorized as
	\begin{align*}
		\bbE \ETp(\cH) = & \prod_{k=1}^{d-1} \left(\sum_{m=0}^{M-1} \etp\left(\frac{m}{M}, \frac{m}{M}\right) - \frac{1}{M-1} \sum_{\substack{m, n=0 \\ m \neq n}}^{M-1} \etp\left(\frac{m}{M}, \frac{n}{M}\right)\right) \frac{\Delta - T}{M} \\
		& + \prod_{k=1}^{d-1} \left(\sum_{m, n = 0}^{M-1} \etp\left(\frac{m}{M}, \frac{n}{M}\right)\right) \frac{T}{M} \\
		= & \frac{M^{d-2}(\Delta-T)^d}{(M-1)^{d-1}} + M^{d-2}T^d
	\end{align*}
\end{proof}

In terms of periodic $L_2$-discrepancy we obtain the following.

\begin{thm} \label{cor:expectation_discrepancy}
	As $\cH$ runs uniformly through all $d$-dimensional weak $M$-Latin hypercubes we get
	\begin{align} \label{eq:expectation_latin_hypercube}
		\bbE \Lp(\cH)^2 = \frac{(M-1)(M+1)^d + (2M^2+1)^d - 2^d M^{2d}}{6^d M^2}.
	\end{align}
	In particular, for $d \geq 4$ there is a point set $X$  with $N$ points and constructed from a weak Latin hypercube with
	\begin{align} \label{ineq:upper_bound}
		\Lp(X) \leq \left(\frac{d}{2 \cdot 3^d}\right)^{1/2} (1+o(1)) N^{\frac{d-2}{d-1}}
	\end{align}
	with $o(1) \longrightarrow 0$ as $\#X \longrightarrow \infty$.
\end{thm}

Combining this with Theorem \ref{cor:Lp_Le_latin_hypercubes} also shows that, for $d \geq 4$ there is a weak Latin hypercube $\cH$ (in fact the same one that fulfils \eqref{ineq:upper_bound}) such that for the corresponding point set $X$ it holds
$$
\Le(X) \leq \left(\frac{d}{12^d}\right)^{1/2} (1+o(1)) N^{\frac{d-2}{d-1}}.
$$

\begin{rem}
	A closer look at the proof of Lemma \ref{lem:probability_distribution} shows that we even have \eqref{eq:expectation_latin_hypercube} for the following randomly generated weak Latin hypercube. Fix a weak Latin hypercube $\cH_0$ and choose $d$ permutations $\sigma_1, \dots, \sigma_d$ uniformly at random. We then get a random weak Latin hypercube by
	$$
	\cH \coloneqq \{(\sigma_1(m_1), \dots, \sigma_d(m_d)): \bm \in \cH_0\}.
	$$
	This random process is much simpler to implement algorithmically and together with \eqref{eq:expectation_latin_hypercube} we see that for every weak Latin hypercube $\cH_0$ there are permutations $\sigma_1, \dots, \sigma_d$ such that the corresponding weak Latin hypercube $\cH$ with point set $X = X(\cH)$ fulfils
	$$
	\Lp(X) \leq \left(\frac{d}{2 \cdot 3^d}\right)^{1/2} (1+o(1)) N^{\frac{d-2}{d-1}}
	$$
	for $d \geq 4$. In this sense, every weak Latin hypercube can be modified via coordinate permutations to get a weak Latin hypercube of asymptotically optimal order.
\end{rem}

Theorem \ref{cor:lower_bound_latin_hypercube_discrepancy} together with Theorem \ref{cor:expectation_discrepancy} shows that for $d \geq 3$ we have the correct asymptotic rate of $N^{\frac{d-2}{d-1}}$ for the optimal periodic $L_2$-discrepancy of weak Latin hypercubes and even the constant factors match for $d \geq 4$.

Compared to the asymptotically optimal rate $\Lp(X) \gtrsim (\log N)^{\frac{d-1}{2}}$ for general sets $X \sbse [0, 1)^d, \#X = N$ (see \cite{HKP20}, based on a method in \cite{Rot54}) as discussed at \eqref{ineq:roth}, we see that for $d \geq 3$ weak Latin hypercubes are unable to reach a globally optimal behaviour in the continuous setting. Permutation sets however seem to be a very good candidate for (approximate) global minimizers of the periodic $L_2$-discrepancy (see the numerical results in \cite{HO16}) and perhaps other notions of discrepancy.

\section{Conclusion and outlook}

Motivated by understanding the utility of permutations for disrepancy theory we introduced permutation sets and observed that they obey the special relation given in Theorem \ref{thm:Lp_Le_permutation}. This was then generalized to weak Latin hypercubes, a class of point sets that we introduced for this purpose, to arbitrary dimensions for more general notions of energies (Proposition \ref{thm:expression_for_excess}) and then applied to the periodic and extreme $L_2$-discrepancy to obtain Theorem \ref{cor:Lp_Le_latin_hypercubes}. Since this gives a lower bound for the periodic $L_2$-discrepancy we aimed for a better understanding of the generalized energy of point sets on the discretized torus (Theorem \ref{thm:svd_energy}) concluding in the desired bounds in Theorem \ref{cor:lower_bound_latin_hypercube_discrepancy}.

There are still some question left for further research. We give some of them here.

\begin{itemize}
	\item [(i)] It turned out that weak Latin hypercubes were the right setting for generalizing Theorem \ref{thm:Lp_Le_permutation} but yielded point sets far from optimality for $L_2$-discrepancy. What is the right generalization for finding (approximate) global minimizers of the periodic $L_2$-discrepancy in dimensions $d \geq 3$? Perhaps Latin hypercubes, which gives sets of the form
	$$
	\left\{\frac{1}{M}(\sigma_1(m), \dots, \sigma_d(m)) : m=0, 1, \dots, M-1\right\}
	$$
	for permutations $\sigma_1, \dots, \sigma_d$ of $\{0, 1, \dots, M-1\}$ are better candidates for this task. While we are not aware of any numerical results in this direction (\cite{HO16} only treats the $2$-dimensionsl case), it is known that such sets can at least be asymptotically optimal. For example shifted digital nets \cite{KP22} and unshifted order $2$ digital $(t, m, d)$-nets over $\bbZ_2$ with regular generator matrices \cite{HMOU16, KP22} give asymptotically optimal point sets of this form. As mentioned in \cite{HMOU16}, there are no know (rational) lattices which achieve Roth's bound for $d \geq 3$ dimensions.
	
	\item [(ii)] Is the result of Proposition \ref{thm:expression_for_excess} best possible in the following sense? Let $w, w' : G \ra \bbR$ be two weights on the discretized torus. We call $w$ and $w'$ \textit{separable} if there is an energy triple $\cT$ so that $\exT(w) \neq \exT(w')$, else \textit{non-separable}. Proposition \ref{thm:expression_for_excess} shows that two weights both of equal constant row sum are non-separable. Are there other cases of non-separable weights?
	
	\item [(iii)] Which other statistical statements can be made about random (weak) Latin hypercubes? For $d=2$, it seems possible to get a concentration result via Chebyshev's inequality using
	$$
	\Var \Lp(\sigma)^2 = \frac{(N-3)(N-2)^2(N-1)^2(N+1)^2}{16200N^5} \sim \frac{N^2}{16200},
	$$
	where $\sigma$ is a randomly chosen (uniformly) permutation of $\{0, 1, \dots, N-1\}$. We have checked this formula via computer calculations for $N \leq 13$ but we did not verify it in detail. We have no suggestion for a similar formula for higher dimensional (weak) Latin hypercubes.
\end{itemize}

\section{Acknowledgement} 

This work was funded by the ESF Plus Young Researchers Group ReSIDA-H2 (SAB, project number 100649391). The author would also like to thank Dmitriy Bilyk for invaluable input and the anonymous referee for many helpful comments undoubtedly improving the quality of the paper. 

\end{document}